\documentclass[10pt,a4paper,oneside,reqno]{article}

\usepackage{amsmath}
\usepackage{amsthm}
\usepackage{mdframed}
\usepackage{setspace}
\usepackage{amsthm}
\usepackage{accents}
\usepackage{amssymb}
\usepackage{graphicx}
\usepackage{color}
\usepackage{enumerate}
\usepackage{framed} 
\usepackage{wasysym}
\usepackage{soul}
\usepackage{indentfirst}
\usepackage[shortlabels]{enumitem}
                    \setlist[enumerate, 1]{1\textsuperscript{o}}
\usepackage{tikz}
\usetikzlibrary{matrix,positioning,decorations.pathreplacing}

\theoremstyle{definition} 
\newtheorem{thm}{Theorem}[section]
\newtheorem{prop}[thm]{Proposition}
\newtheorem{lem}[thm]{Lemma}
\newtheorem{conj}[thm]{Conjecture}
\newtheorem{cor}[thm]{Corollary}
\newtheorem{defn}[thm]{Definition}
\newtheorem{ex}[thm]{Example}
\newtheorem{notn}[thm]{Notation}
\newtheorem{rmk}[thm]{Remark}

\begin{document}

%%%%%%%%%%%%%%%%%%%%%%%%%%%%%%%%%%%%%%%%%%%%%%%%%%%%%%%%%%%%%%%%%%%%%%%%%%%%%%%%%%%%%%%%%%%%%%%%%%%%%%%%%%%%%%%%

\title{Towards a conjecture of Pappas and Rapoport on a scheme attached to the symplectic group}
\author{Hanveen Koh\thanks{\texttt{hkoh5@jhu.edu.} Department of Mathematics, Johns Hopkins University, Baltimore, MD 21218, US}}
\date{}
\maketitle

\begin{abstract}
Let $n=2r$ be an even integer. We consider a closed subscheme $V$ of the scheme of $n \times n$ skew-symmetric matrices, on which there is a natural action of the symplectic group $Sp(n)$. Over a field $F$ with $\mbox{char }F \neq 2$, the scheme $V$ is isomorphic to the scheme appearing in a conjecture by Pappas and Rapoport on local models of unitary Shimura varieties. With the additional assumption $\mbox{char }F = 0$ or $\mbox{char }F > r$, we prove the coordinate ring of $V$ has a basis consisting of products of pfaffians labelled by King's symplectic standard tableaux with no odd-sized rows. When $n$ is a multiple of 4, the basis can be used to show that the coordinate ring of $V$ is an integral domain, and this proves a special case of the conjecture by Pappas and Rapoport.
\end{abstract}

%%%%%%%%%%%%%%%%%%%%%%%%%%%%%%%%%%%%%%%%%%%%%%%%%%%%%%%%%%%%%%%%%%%%%%%%%%%%%%%%%%%%%%%%%%%%%%%%%%%%%%%%%%%%%%%%

\section{Introduction}

Let $F$ be an infinite field and $n$ be a positive integer. It's well known that for each partition $\lambda = (\lambda_1, \lambda_2, \dotsc, \lambda_k)$ of some positive integer with $0 < \lambda_k \leq \cdots \leq \lambda_1 \leq n$, there is an associated $GL(n, F)$-module such that the bideterminants labelled by the standard tableaux of shape $\lambda$ form a basis of the module. This module is called the Weyl module and is irreducible when $F$ is of characteristic zero. 

Assume $n$ is an even integer. For the symplectic case, Berele [1] has given a basis of the irreducible $Sp(n, F)$-module, called the symplectic Weyl module, over a field $F$ of characteristic zero where the basis is labelled by the symplectic standard tableaux (defined by King [7]; see Definition 3.1). In [5], Donkin proves this result for arbitrary infinite field. Donkin shows by a symplectic version of the Carter-Lusztig Lemma that the bideterminants indexed by the symplectic standard tableaux of shape $\lambda$ form a basis of the $Sp(n, F)$-module, which is defined as the span of all bideterminants associated to the tableaux of shape $\lambda$. (This module need not be irreducible anymore.)

In this article, we are interested in a scheme defined as follows:
Let $R$ be a commutative ring with unity and let $n =2r$ be an even integer. Let $V$ be the scheme of $n \times n$ matrices $Y = (Y_{i j})$ over $\mbox{Spec} \: R$ such that 
\[
Y = -Y^T,  \quad Y_{ii} = 0 \: \text{ for }\: i = 1, \dotsc, n,  \quad Y^T JY = 0, \quad \mbox{and} \quad \mbox{char}_{(-JY)} (T) = T^n, \tag{1.1} \label{1.1}
\]
where $J$ is the $n \times n$ matrix
\[
\renewcommand{\arraystretch}{1.5}
J =
\left[
\begin{array}{ccc|ccc}
& & & \: 1 \: & & \\
& & & & \: \ddots \: & \\
& & & & & \: 1 \:\\
\hline
-1 & & & & & \\
& \ddots & & & & \\
& & -1 & & & 
\end{array}
\right].  \tag{1.2} \label{1.2}
\]

\noindent
The first two conditions in $(\ref{1.1})$ implies that $V$ is a closed subscheme of the scheme of $n \times n$ skew-symmetric matrices over $\mbox{Spec} \: R$. (The second condition is redundant if $2 \in R^{\times}$.)  If a tableau has no odd-sized rows, we call it an \emph{even-tableau} (and \emph{even-tableaux} for the plural form). In [4, Ch.6], De Concini and Procesi show that there is an $R$-basis of the ring $R[Y_{i j}] \big/ (Y + Y^T, Y_{11}, \dotsc, Y_{nn})$ indexed by the standard even-tableaux, where each such tableau corresponds to a product of pfaffians. On the other hand, there is a natural action of the symplectic group $Sp(n) := \{ g \: | \:g^T J g = J \}$ on $V$ by  $Y \cdot g = g^{T}Yg$. One aim of this article is to find an $R$-module basis for the coordinate ring of $V$, denoted by $R[V]$, in terms of tableaux when $R$ is given some suitable conditions. We omit the case $n=2$ since $R[V]$ is isomorphic to $R$ in this case given that $2 \in  R^{\times}$. In Section 3 we prove the following:

\begin{thm}  %% 1.1
Let $n=2r$ be an even integer with $n \geq 4$. When the scheme $V$ is defined over a field $F$ with $\mbox{char } F = 0$ or $\mbox{char } F > r$, there is an $F$-basis for the coordinate ring of $V$ consisting of products of pfaffians labelled by the symplectic standard even-tableaux.
\end{thm}

In Section 2 we develop some relations between pfaffians for later use. In Section 3 we define the symplectic standard even-tableaux and show that they can be used to label a basis of the coordinate ring of $V$. In Section 4 we prove a special case of Pappas and Rapoport's conjecture [8, Conj.5.2] when $n$ and $F$ meet some additional conditions.

\begin{conj}  %% 1.2
(Pappas and Rapoport, 2009) Let $F$ be a field with $\mbox{char } F \neq 2$ and let $n$ be an integer divisible by 4. Let $W$ be the scheme of $n \times n$ matrices $X$ over $\mbox{Spec} \: F$ with
\[
-JX^T J = X, \quad X^2 = 0,  \quad \mbox{and} \quad \mbox{char}_X (T) = T^n.   \tag{1.3} \label{1.3}
\]
Then $W$ is reduced.
\footnote{In fact, Pappas and Rapoport formulate a more general conjecture for any even $n$, depending on a partition of $n$ into two even parts; the version we have stated is the case $r = s$ in their notation.}
\end{conj}

\noindent
Assume that $n$ is divisible by 4. Taking $R = F$, there is an isomorphism $W \rightarrow V$ given by $X \mapsto JX$. When $F$ is a field with $\mbox{char } F = 0$ or $\mbox{char } F > r$, we prove that the coordinate ring $F[V]$ is in fact an integral domain (Theorem 4.3) so the same holds true for $W$ as well.

\begin{thm}  %% 1.3
Let $n$ be a multiple of 4 and $F$ be a field with $\mbox{char } F = 0$ or $\mbox{char } F > r$. Let $W$ be the scheme defined as in Conjecture 1.2. Then the coordinate ring of $W$ is an integral domain.
\end{thm}

%%%%%%%%%%%%%%%%%%%%%%%%%%%%%%%%%%%%%%%%%%%%%%%%%%%%%%%%%%%%%%%%%%%%%%%%%%%%%%%%%%%%%%%%%%%%%%%%%%%%%%%%%%%%%%%%

\section{The Relations}

Let $n = 2r$ be an even integer with $n \geq 4$ and $F$ be a field. We denote by $\bf{n}$ the set $\{1, 2, \dotsc, n\}$ and, for $1 \leq i \leq r$, we often use the symbol $\overline{i}$ in place of $r+i$. That is,
\[ 
\mathbf{n} = \{1, 2, \dotsc, r, \overline{1}, \overline{2}, \dotsc, \overline{r}\}. 
\]

\noindent
Let $C$ be any $F$-algebra. Let $E$ be a free $C$-module of rank $n$ with a basis $\{ e_1, \dotsc, e_r, e_{\overline{1}}, \dotsc, e_{\overline{r}}\}$. We endow $E$ with a nondegenerate antisymmetric bilinear form $\langle \:,\: \rangle$ such that
\[ 
\langle e_i, e_{\overline{j}} \rangle = \delta_{ij} = - \langle e_{\overline{j}}, e_i \rangle, \quad \langle e_i, e_j \rangle = 0, \quad \langle e_{\overline{i}}, e_{\overline{j}} \rangle = 0
\]
for any $i, j \in \{1, \dotsc, r\}.$ In other words, $\langle \:,\: \rangle$ is represented by the matrix $J$, $( \ref{1.2} )$, with respect to the basis  $\{ e_1, \dotsc, e_r, e_{\overline{1}}, \dotsc, e_{\overline{r}}\}$.

For each $1 \leq k \leq r$ and $t \leq \lfloor k/2 \rfloor $, we introduce a homomorphism
\[ \Phi_{k, t}: \bigwedge^k E \rightarrow \bigwedge^{k-2t} E \]
defined by
\[ \Phi_{k, t} (v_1 \wedge v_2 \wedge \cdots \wedge v_k) = \sum_\sigma sgn(\sigma) \langle v_{\sigma(1)}, v_{\sigma(2)}\rangle \cdots \langle v_{\sigma(2t-1)}, v_{\sigma(2t)} \rangle \: v_{\sigma(2t+1)} \wedge \cdots \wedge v_{\sigma(k)}, \]
where $\sigma$ runs through all permutations of $\{1, 2, \dotsc, k\}$ such that 
\[ \sigma(1) < \sigma(2), \quad  \dotsc, \quad \sigma(2t-1) < \sigma(2t), \quad \mbox{and} \quad \sigma(2t+1) < \sigma(2t+2) < \cdots < \sigma(k). \]

\noindent
It is easy to see that $\Phi_{k, t}$ is well-defined, and in fact, $\Phi_{k, t}$ can be obtained recursively:
\[
\Phi_{k, t} = \Phi_{k-2t+2, 1} \circ \cdots \circ \Phi_{k-2, 1} \circ \Phi_{k, 1}. \tag{2.1} \label{2.1}
\]
$\Phi_{k, t}$ is the main tool to find relations between pfaffians in the coordinate ring  of $V$, cf. [3, p.5]. 
% notice however that $sgn(\sigma)$ is defined differently in [2].
One goal of this section is to achieve the relation $(\ref{2.5})$ below, which is the analog of the equation given in [3, Prop.1.8] if we substitute pfaffians for minors (with fixed column indices). Some of the proofs and definitions in this section closely follow [3, pp.5-8].\\

\begin{lem}  %% 2.1
For any $n \times n$ matrix $Y$ in $V(C)$, we have
$$ 2 \sum_{h=1}^r Y_{h \overline{h}} = 0.$$
\end{lem} 

\begin{proof}
$-JY$ has trace 0 since $\mbox{char}_{(-JY)} (T) = T^n$. Then
\[
0 = \mbox{tr}(-JY) = - \sum_{h=1}^r Y_{\overline{h} h} +  \sum_{h=1}^r Y_{h \overline{h}} = 
- \sum_{h=1}^r (-Y_{h \overline{h}}) +  \sum_{h=1}^r Y_{h \overline{h}} =
2 \sum_{h=1}^r Y_{h \overline{h}} \: .
\]
\end{proof}

\begin{lem}  %% 2.2
Assume $k=2m$ is an even number and $Y \in V(C)$. Let
$$ w := \sum_{i, j \in \mathbf{n}} Y_{ij} \: e_i \wedge e_j .$$ 
Then  $\Phi_{k, t}(w^m) = 0$ for any $t \leq m$.  Here the $m$-th power is taken in the exterior algebra $\bigwedge E$.\\
\end{lem} 

\begin{proof}
It suffices to show $\Phi_{k, 1}(w^m) = 0$ by $( \ref{2.1} )$.

\noindent \textit{Case}
 $m = 1$: By Lemma 2.1,
\begin{flalign*}
\Phi_{2,1} (w)
&= \Phi_{2,1} (\sum_{i, j \in \mathbf{n}} Y_{ij} \: e_i \wedge e_j )\\
&= \sum_{i, j \in \mathbf{n}} Y_{ij} \langle e_i, e_j \rangle\\
&= \sum_{h=1}^{r} Y_{h \overline{h}} \langle e_h, e_{\overline{h}} \rangle +  \sum_{h=1}^{r} Y_{\overline{h} h} \langle e_{\overline{h}}, e_h \rangle\\
&= 2\sum_{h=1}^{r} Y_{h \overline{h}}\\
&= 0.
\end{flalign*}

\noindent \textit{Case}
 $m = 2$ :
\begin{flalign*}
\Phi_{4,1} (w^2)
&= \Phi_{4,1} \bigg( (\sum_{i, j \in \mathbf{n}} Y_{ij} e_i \wedge e_j )(\sum_{i', j' \in \mathbf{n}} Y_{i'j'} e_{i'} \wedge e_{j'} ) \bigg)\\
&= \sum_{i, j, i', j' \in \mathbf{n}} Y_{ij} Y_{i'j'} \Phi_{4, 1} (e_i \wedge e_j \wedge e_{i'} \wedge e_{j'}). \tag{2.2} \label{2.2}
\end{flalign*}

\noindent
By definition,
\begin{flalign*}
&\Phi_{4, 1} (e_i \wedge e_j \wedge e_{i'} \wedge e_{j'}) \\
& \hspace{5mm} = \langle e_i, e_j \rangle e_{i'} \wedge e_{j'}
 + \langle e_{i'}, e_{j'} \rangle e_{i} \wedge e_{j}
- \langle e_i, e_{i'} \rangle e_{j} \wedge e_{j'} 
- \langle e_j, e_{j'} \rangle e_{i} \wedge e_{i'}
+ \langle e_i, e_{j'} \rangle e_{j} \wedge e_{i'}
+ \langle e_j, e_{i'} \rangle e_{i} \wedge e_{j'}.
\end{flalign*}

\noindent
Then we compute the sum $( \ref{2.2} )$ with each of these six terms. It turns out that each sum is equal to zero, so we get $\Phi_{4,1} (w^2) = 0$.  We use Lemma 2.1 for the first one and the second one. The first sum is

\begin{flalign*}
 \sum_{i, j, i', j'} Y_{ij} Y_{i'j'} \langle e_i, e_j \rangle e_{i'} \wedge e_{j'} 
&= \sum_{i', j'} \bigg( \sum_{i, j} Y_{ij} \langle e_i, e_j \rangle \bigg) Y_{i'j'} \: e_{i'} \wedge e_{j'}\\
&= \sum_{i', j'} \bigg( 2\sum_{h=1}^{r} Y_{h \overline{h}} \bigg)  Y_{i'j'} \: e_{i'} \wedge e_{j'}\\
&= 0,
\end{flalign*}

\noindent
and the second sum can similary be shown to be zero.

For the third one, we use the condition $Y^T JY = 0$:
\begin{flalign*}
\sum_{i, j, i', j'} Y_{ij} Y_{i'j'} \langle e_i, e_{i'} \rangle e_{j} \wedge e_{j'} 
&= \sum_{j, j'}\bigg(\sum_{i, i'} Y_{ij} \langle e_i, e_{i'} \rangle Y_{i'j'} \bigg)  e_{j} \wedge e_{j'}\\
&= \sum_{j, j'} \bigg( \sum_{i, i'} (Y^{T})_{ji} \langle e_i, e_{i'} \rangle Y_{i'j'} \bigg)  e_{j} \wedge e_{j'}\\
&= \sum_{j, j'} \hspace{2mm} (Y^{T} J Y)_{jj'} \hspace{2mm} e_{j} \wedge e_{j'}\\
&= 0.
\end{flalign*}

\noindent
Since $Y$ is a skew-symmetric matrix, i.e. $Y = -Y^T$, we also have  
$0 = Y^T JY = YJY^T = Y^T J Y^T = YJY$,
which we make use of to show that the fourth, the fifth, and the sixth sums are again equal to zero. The computations are analogous to the third one and listed below:

\begin{flalign*}
\sum_{i, j, i', j'} Y_{ij} Y_{i'j'} \langle e_j, e_{j'} \rangle e_{i} \wedge e_{i'}
&= \sum_{i, i'} \bigg( \sum_{j, j'} Y_{ij} \langle e_j, e_{j'} \rangle Y_{i'j'} \bigg) e_{i} \wedge e_{i'}\\
&= \sum_{i, i'} \bigg( \sum_{j, j'} Y_{ij} \langle e_j, e_{j'} \rangle (Y^{T})_{j'i'} \bigg)  e_{i} \wedge e_{i'}\\
&= \sum_{i, i'} \hspace{2mm} (YJY^T)_{ii'} \:  e_{i} \wedge e_{i'}\\
&= 0,\\\\
\sum_{i, j, i', j'} Y_{ij} Y_{i'j'} \langle e_i, e_{j'} \rangle e_{j} \wedge e_{i'}
&= \sum_{j, i'} \bigg( \sum_{i, j'} Y_{ij} \langle e_i, e_{j'} \rangle Y_{i'j'} \bigg) e_{j} \wedge e_{i'}\\
&= \sum_{j, i'} \bigg( \sum_{i, j'} (Y^T)_{ji} \langle e_i, e_{j'} \rangle (Y^T)_{j'i'} \bigg) e_{j} \wedge e_{i'}\\
&= \sum_{j, i'} \hspace{2mm} (Y^T J Y^T)_{ji'}  \:  e_{j} \wedge e_{i'}\\
&= 0,\\\\
\sum_{i, j, i', j'} Y_{ij} Y_{i'j'} \langle e_j, e_{i'} \rangle e_{i} \wedge e_{j'}
&= \sum_{i, j'} \bigg( \sum_{j, i'} Y_{ij} \langle e_j, e_{i'} \rangle Y_{i'j'} \bigg)  e_{i} \wedge e_{j'}\\
&= \sum_{i, j'} \hspace{2mm} (YJY)_{ij'} \: e_{i} \wedge e_{j'}\\
&= 0.\\
\end{flalign*}

\noindent
\textit{Case} $m \geq 2$ :
The general case can be achieved from the first two cases as
\begin{flalign*}
\Phi_{k, 1}(w^m)
&= \binom{m}{2} \bigg\{ \Phi_{4, 1} (w^2) \wedge w^{m-2} -2 \cdot \Phi_{2, 1} (w) \cdot w^{m-1} \bigg\} 
+ m \cdot \Phi_{2, 1} (w) \cdot w^{m-1}\\
&= \binom{m}{2} \Phi_{4, 1} (w^2) \wedge w^{m-2} + (-m^2 +2m) \cdot \Phi_{2, 1} (w) \cdot w^{m-1}\\
&= 0.
\end{flalign*}
\end{proof}

\begin{notn}  %% 2.3
When  $I = \{i_1,  i_2, \dotsc, i_{l}\}$ is a subset of $\mathbf{n}$ with $i_1 < i_2 < \cdots < i_{l}$, let $e^I$ denote the vector
$$ e^I := e_{i_1} \wedge e_{i_2} \wedge \cdots \wedge e_{i_l} \, .$$

\noindent
We define a basis for $\bigwedge^{k} E$ as in [2, p.5]. Let $P = \{ p_1, \dotsc, p_s \}$ and $Q = \{ q_1, \dotsc, q_{k-s} \}$ be two subsets of $ \{1, \dotsc, r \}$ with $|P| + |Q| = k$. Let $\overline{Q}$ denote the set $\{ \overline{q_1}, \dotsc, \overline{q_{k-s}}\}$. A vector $e^{P, \overline{Q}} \in \bigwedge^{k} E$ is defined as follows:

\noindent
(A) If $P \cap Q = \emptyset$, then $e^{P, \overline{Q}} := e^P \wedge e^{\overline{Q}}$.\\
(B) If $P \cap Q = \Gamma$ where $\Gamma = \{\gamma_1 < \cdots < \gamma_\lambda \}$, then $e^{P, \overline{Q}} := e_{\gamma_1} \wedge e_{\overline{\gamma_1}}  \wedge \cdots \wedge e_{\gamma_{\lambda}} \wedge e_{\overline{\gamma_{\lambda}}} \wedge e^{P \smallsetminus \Gamma} \wedge e^{\overline{Q \smallsetminus \Gamma}}$.\\

\noindent
Clearly $\{ e^{P, \overline{Q}} \: | \: P, Q \subset \{1, \dotsc, r\}, \: |P| + |Q| = k \}$ forms a basis of $\bigwedge^k E$.\\
\end{notn}

\begin{ex}  %% 2.4
Let $P = \{ 1, 2, 5 \}$ and $Q = \{2, 3, 4\}$. Then
$$ e^{P, \overline{Q}} = e^{ \{ 1, 2, 5 \}, \{ \overline{2}, \overline{3}, \overline{4} \} } = e_2 \wedge e_{\overline{2}} \wedge e_1 \wedge e_5 \wedge e_{\overline{3}} \wedge e_{\overline{4}} \,.$$
\end{ex}

\begin{lem}  %% 2.5
\begin{enumerate}[(1)]
\item If $t > | P \cap Q|$ then 
\[ 
\Phi_{k, t}(e^{P, \overline{Q}}) = 0.
\]

\item If $t \leq | P \cap Q |$ then
\[
 \Phi_{k, t}(e^{P, \overline{Q}}) = t! \sum_{\Gamma_t} e^{P \smallsetminus {\Gamma_t}, \: \overline{ Q \smallsetminus {\Gamma_t}}}
\]
\noindent where $\Gamma_t$ runs through all size $t$ subsets of $\Gamma = P \cap Q$.
\end{enumerate}
\end{lem}

\begin{proof}
Both assertions are easily seen from the definitions; see [3, Lemma 1.6].
\end{proof}

As in [3, p.6], we follow the convention of putting $e^{P \smallsetminus {\Gamma_t},\: \overline{Q \smallsetminus {\Gamma_t}}} = 0$ when $\Gamma_t \not\subset P \cap Q $. Hence, we can write
\[
 \Phi_{k, t}(e^{P, \overline{Q}}) = t! \sum_{ |\Gamma_t| = t} e^{P \smallsetminus {\Gamma_t}, \: \overline{ Q \smallsetminus {\Gamma_t}}} \, .
\]

%%%%%%%%%%%%%%%%%%%%%%%%%%%%%%%%%%%%%%%%%%%%%%%%%%%%%%%%

Let $A = (A_{ij})$ be a $2s \times 2s$ skew-symmetric matrix. The pfaffian of $A$ is a polynomial $Pf (A)$ in the entries of A defined by
\[
Pf (A) = \sum_{\sigma} sgn(\sigma) \prod_{i=1}^{s} A_{\sigma(2i-1) \sigma(2i)},
\]
where the summation is over all permutations $\sigma$ of $\{ 1, 2, \dotsc, 2s \}$ such that
\[
\sigma(1) < \sigma(3) < \cdots < \sigma(2s-1) \quad \mbox{and} \quad \sigma(2i-1) < \sigma(2i) \mbox{ for } i = 1, 2, \dotsc, s.
\]
It's well known that $Pf(A)^2 = \mbox{det} (A)$. If $A$ has entries in $C$ and $\{e_1, e_2, \dotsc, e_{2s}\}$ is the standard basis of $C^{2s}$, then $Pf(A)$ satisfies
\[
 \big( \sum_{i < j} A_{ij} \; e_i \wedge e_j \big)^s = s! \: Pf(A) \: e_1 \wedge e_2 \wedge \cdots \wedge e_{2s};
\]
see [2, \S 5.2].

For a given $Y \in V(C)$, let $[ i_1,  i_2, \dotsc, i_{2l}]$ denote the pfaffian of the principal submatrix of $Y$ indexed by $i_1,  i_2, \dotsc, i_{2l} \in \mathbf{n}$. Then for any $m \leq r$, we have
\[
 (\sum_{i, j \in \mathbf{n}} Y_{ij} \: e_i \wedge e_j )^m
= 2^m (\sum_{\substack{ i, j \in \mathbf{n} \\  i < j}} Y_{ij} \: e_i \wedge e_j )^m
= 2^m m! \sum_{ i_1 <  i_2 < \cdots < i_{2m}} [ i_1,  i_2, \dotsc, i_{2m}] \: e_{i_1} \wedge e_{i_2} \wedge \cdots \wedge e_{i_{2m}}
\]
where the rightmost sum is among all size $2m$ subsets of $\mathbf{n}$. Recall that $e^{P, \overline{Q}}$ is a wedge product of $e_i$, $i \in P \cup \overline{Q}$, but the indices are not necessarily in increasing order (see Notation 2.3 (B)). It's convenient to define $[P, \overline{Q}]$ likewise so that we have the identity
\[
(\sum_{i, j \in \mathbf{n}} Y_{ij} \: e_i \wedge e_j )^m = 2^m m! \sum_{\substack{P, Q  \subset \{1, \cdots, r\} \\  |P| + |Q| = 2m}} [P, \overline{Q}] \: e^{P, \overline{Q}} \, . \tag{2.3} \label{2.3} 
\]

\noindent
That is, $[P, \overline{Q}]$ denotes the pfaffian of the $2m \times 2m$ principal submatrix of $Y$ indexed by $P \cup \overline{Q}$, but the order of indices coincides with that of $e^{P, \overline{Q}}$.

\begin{ex}  %% 2.6
Let $P = \{ 1, 2, 5 \}$ and $Q = \{2, 3, 4\}$. Then
$$ [P, \overline{Q}] = [ \{ 1, 2, 5 \},  \{ \overline{2}, \overline{3}, \overline{4} \} ] = [2, \overline{2}, 1, 5, \overline{3}, \overline{4}]$$

\noindent
( $= -[1, 2, 5, \overline{2}, \overline{3}, \overline{4}]$ since $( 2, \overline{2}, 1, 5, \overline{3}, \overline{4} )$ is an odd permutation of $( 1, 2, 5, \overline{2}, \overline{3}, \overline{4} )$).
\end{ex}

\begin{lem}  %% 2.7
Assume $k = 2m$ is an even integer, and $ 1 \leq t \leq m \leq r$. Let $P'$ and $Q'$ be two fixed subsets of $\mathbf{r} := \{ 1, 2, \dotsc, r \}$ with $|P'| + |Q'| = k-2t$. Then for any $Y \in V(C)$,
\[
2^m m! \hspace{0.3mm} t! \sum_{\Gamma_t} [P' \cup \Gamma_t, \overline{Q' \cup \Gamma_t}] = 0
\]
where $\Gamma_t$ runs through all size $t$ subsets of $\mathbf{r} \smallsetminus ( {P' \cup Q'} )$. 

In particular, if $F$ is a field with $\mbox{char } F = 0$ or $\mbox{char } F > r$ then
\[
 \sum_{\substack{\Gamma_t  \subset \hspace{1mm} \mathbf{r} \smallsetminus (P' \cup Q')\\  { |\Gamma_t| = t}}} [P' \cup \Gamma_t, \overline{Q' \cup \Gamma_t}] = 0. \tag{2.4} \label{2.4}
\]
\end{lem}

\begin{proof}
This lemma is based on [3, Prop.1.7] and follows a similar proof structure. 
By Lemma 2.2 and equation $(\ref{2.3})$, we have
\[
0 =  \Phi_{k, t} \bigg( (\sum_{i, j \in \mathbf{n}} Y_{ij} e_i \wedge e_j )^m \bigg) 
= \Phi_{k, t} \bigg( 2^m m!  \sum_{\substack{P, Q  \subset  \mathbf{r} \\  |P| + |Q| = k}} [P, \overline{Q}] \: e^{P, \overline{Q}} \bigg).
\]

\noindent Then
\begin{flalign*}
\Phi_{k, t} \bigg( 2^m m!  \sum_{\substack{P, Q  \subset  \mathbf{r} \\  |P| + |Q| = k}} [P, \overline{Q}] \: e^{P, \overline{Q}} \bigg)
&= 2^m m! \sum_{\substack{P, Q  \subset  \mathbf{r} \\  |P| + |Q| = k}} [P, \overline{Q}] \: \Phi_{k, t} (e^{P, \overline{Q}})\\
&= 2^m m! \sum_{\substack{P, Q  \subset   \mathbf{r} \\  |P| + |Q| = k}} [P, \overline{Q}] \bigg( t! \sum_{ |\Gamma_t| = t} e^{P \smallsetminus {\Gamma_t}, \hspace{1mm} \overline{Q \smallsetminus {\Gamma_t}}} \bigg) \\
&=  2^m m! \hspace{0.3mm} t! \sum_{\substack{P', Q'  \subset   \mathbf{r} \\  |P'| + |Q'| = k-2t}} \bigg(  \sum_{\substack{\Gamma_t  \subset \hspace{1mm} \mathbf{r} \smallsetminus (P' \cup Q')\\  { |\Gamma_t| = t}}} [P' \cup \Gamma_t, \overline{Q' \cup \Gamma_t}] \bigg) e^{P', \overline{Q'} }
\end{flalign*}

\noindent
by putting $P' = P \smallsetminus {\Gamma_t}$,  $Q' = Q \smallsetminus {\Gamma_t}$ and changing the order of summation. Since $  e^{P', \overline{Q'} }$ are linearly independant, this proves the lemma. When $\mbox{char } F =0$ or $\mbox{char } F > r$, $2^m m! \hspace{0.3mm} t! \in C^{\times}$ so we get ($\ref{2.4}$).
\end{proof}

For the rest of this section, we assume that $F$ is a field with $\mbox{char } F = 0$ or $\mbox{char } F > r$.

\begin{prop}  %% 2.8
Let $ [P' \cup \Gamma, \overline{Q' \cup \Gamma}] $ be a fixed pfaffian of $Y \in V(C)$ with $\Gamma \subset \{1, 2, \dotsc, r\} \smallsetminus (P' \cup Q')$ and $\Gamma \neq \varnothing$. ($P'$ and $Q'$ are not necessarily disjoint.) Then
\[ 
[P' \cup \Gamma, \overline{Q' \cup \Gamma}] - (-1)^{ |\Gamma| } \sum_{\Gamma'} [P' \cup \Gamma', \overline{Q' \cup \Gamma'}] = 0 \tag{2.5} \label{2.5}
\]

\noindent where $\Gamma'$ runs over the subsets of $\{1, 2, \dotsc, r\} \smallsetminus (P' \cup Q' \cup \Gamma$) with $ |\Gamma'| = |\Gamma|$.
\end{prop}

\begin{proof}
The proof is essentially the same as that of [3, Prop.1.8] if we substitute pfaffians for the minors (with fixed column indices $h_1, \dotsc, h_k$) and apply $( \ref{2.4} )$ as a replacement for [3, (1.7)].
\end{proof}

\begin{cor}  %% 2.9
For $Y \in V(C)$, any pfaffian $[P, \overline{Q}]$ with $|P| + |Q| > r$ vanishes.
\end{cor}

\begin{proof}
Let $[P, \overline{Q}]$ be such a pfaffian. Then clearly $P \cap Q \neq \varnothing$. We define
$$ \Gamma = P \cap Q, \quad P' = P \smallsetminus \Gamma, \quad Q' = Q \smallsetminus \Gamma,$$
so that $[P, \overline{Q}] =  [P' \cup \Gamma, \overline{Q' \cup \Gamma}]$. Since $|P'| + |Q'| + 2|\Gamma| > r$, there is no $\Gamma' \subset \{1, 2, \dotsc, r\} \smallsetminus (P' \cup Q' \cup \Gamma)$ such that $|\Gamma'| = |\Gamma|$. By Proposition 2.8, we get 
$$[P' \cup \Gamma, \overline{Q' \cup \Gamma}] = 0.$$
\end{proof}

%%%%%%%%%%%%%%%%%%%%%%%%%%%%%%%%%%%%%%%%%%%%%%%%%%%%%%%%
%%%%%%%%%%%%%%%%%%%%%%%%%%%%%%%%%%%%%%%%%%%%%%%%%%%%%%%%

\section{Symplectic standard even-tableaux}

We define a tableau as in [5, p.117]. Let $N$ be a positive integer and $\lambda = (\lambda_1, \lambda_2, \dotsc, \lambda_k)$ be a partition of $N$, i.e.
\[
\lambda_1 \geq \lambda_2 \geq \cdots \geq \lambda_k > 0 \quad \mbox{ and } \quad \sum_{i=1}^{k} \lambda_i = N.
\]
\noindent
The \emph{diagram} $D(\lambda)$ of $\lambda$ is defined as the set $\{ (s, t) \in \mathbb{Z} \times \mathbb{Z} \: | \: 1 \leq s \leq k, \: 1 \leq t \leq \lambda_s \}$. Recall that $\bf{n}$ denotes the set $\{1, 2, \dotsc, r, \overline{1}, \overline{2}, \dotsc, \overline{r}\}$. A \emph{tableau} of shape $\lambda$ with entries in $\bf{n}$ is a map $T: D(\lambda) \rightarrow \bf{n}$, depicted by its array of values

\[
\begin{array}{ccccc}
T(1, 1) & T(1, 2) &  \cdots &  \cdots & T(1, \lambda_1) \\
T(2, 1) & T(2, 2) &  \cdots &  T(2, \lambda_2) & \\
T(3, 1) & T(3, 2) &  \cdots & & \\
\vdots & \vdots & & & 
\end{array}
\]

\noindent
In this document, tableaux have entries in $\bf{n}$ unless specified otherwise. We order the indices of $\bf{n}$ by
\[ \overline{1} \prec 1 \prec \overline{2} \prec 2 \prec \cdots \prec \overline{r} \prec r.  \tag{3.1}  \label{3.1}\]
(This is different from the natural numerical order.) A tableau $T$ is called \emph{standard} if its entries strictly increase along each row and weakly increase down each column with respect to the order $( \ref{3.1} )$. Now we give the definition of \emph{symplectic standard} tableaux following King [7]; note however that the role of rows and columns are exchanged in definitions of standard tableaux and symplectic standard tableaux in [5] and [7]. (De Concini also defines symplectic standard tableaux in [3], but that definition is different from the one given here.)

\begin{defn}  %% 3.1
A tableau $T$ is called \emph{symplectic standard} if
\begin{itemize}
\item it is standard, and
\item the indices $p$ and $\overline{p}$ appear only in the first $p$ columns for $1 \leq p \leq r$.
\end{itemize}
\end{defn}

\noindent
Each row of a symplectic standard tableau must have length $\leq r$ by definition.

\begin{defn}  %% 3.2
If a tableau has no odd-sized rows, we call it an \emph{even-tableau} (and \emph{even-tableaux} for the plural form). In other words, if $T$ is an even-tableau of shape $\lambda = (\lambda_1, \lambda_2, \dotsc, \lambda_k)$, then all $\lambda_i$ are even.
\end{defn}

 Let $R$ be a commutative ring with unity and let $Y = (Y_{ij})$ be an $n \times n$ skew-symmetric matrix of indeterminates, i.e.
\begin{itemize}
\item $Y_{ij} = -Y_{ji} \: \mbox{ if } i < j$, and
\item $Y_{ii} = 0 \mbox{ for all } i$.
\end{itemize}
Let $R[Y_{ij}]_{i<j}$ denote the polynomial ring $R[Y_{ij}: 1 \leq i < j \leq n]$. When $\lambda = (\lambda_1, \lambda_2, \dotsc, \lambda_k)$ is a partition of an even integer $2m$ such that all $\lambda_i$ are even, we associate a tableau $T$ of shape $\lambda$ with the product of pfaffians evaluated at the generic skew-symmetric matrix $Y$, denoted by $[T]$,
$$ [T] := \prod_{i=1}^k [T(i, 1), T(i, 2), \dotsc, T(i, \lambda_i) ]. $$
Clearly $[T]$ can be considered as a degree $m$ homogeneous polynomial in $R[Y_{ij}]_{i<j}$. The combinatorial structure of $R[Y_{ij}]_{i<j}$ is observed by the following facts; see [4, Ch.6] for the proof.

\begin{thm}  %% 3.3
(De Concini and Procesi, 1976) 
\begin{enumerate}[(1)]
\item The products of pfaffians indexed by the standard even-tableaux form an $R$-basis of $R[Y_{ij}]_{i<j}$.
\item For any given even-tableau $T$, the standard representation of $[T]$ can be achieved by iterated use of the following relation
\end{enumerate}
\begin{flalign*}
[a_1, \dotsc, a_{2t}] [b_1, \dotsc, b_{2s}] - \sum_{i=1}^{2t} [a_1, \dotsc, a_{i-1}, b_1, a_{i+1}, \dotsc, a_{2t}] [a_i, b_2, \dotsc, b_{2s}] \\
= \sum_{j=2}^{2s} (-1)^{j-1} [b_2, \dotsc, b_{j-1}, b_{j+1}, \dotsc, b_{2s}] [b_j, b_1, a_1, \dotsc, a_{2t}].
\end{flalign*}
\end{thm}

\begin{rmk}  %% 3.4
Consider the natural numerical order given on the indices of $\bf{n}$:
\[
1< 2 < \cdots < r < \overline{1} < \overline{2} < \cdots < \overline{r}.   \tag{3.2}  \label{3.2}
\]
Recall that the definition of the standard tableaux depends on the order $(\ref{3.1})$ given on $\bf{n}$. There is another $R$-basis of $R[Y_{ij}]_{i<j}$ indexed by '$<$-standard' even-tableaux. More precisely, the map $T \mapsto [T]$ from 
\[
\left\lbrace T \; \middle| \;
\begin{tabular}{@{}l@{}}
$T$ is an even-tableau such that the entries of $T$ strictly increase along each row \\
and weakly increase down each column with respect to the order $(\ref{3.2})$
\end{tabular}
\right\rbrace
\]
into $R[Y_{ij}]_{i<j}$ is injective and its image forms an $R$-basis of $R[Y_{ij}]_{i<j}$.
\end{rmk}

The \emph{type} of a tableau $T$ is defined as the $2r$-tuple of integers $(a_1, b_1, \dotsc, a_r, b_r)$ where
 
\begin{center}
$a_p := |T^{-1}(\overline{p})|$ = the number of occurences of $\overline{p}$ in $T$\\
$b_p := |T^{-1}(p)|$ = the number of occurences of $p$ in $T$
\end{center}

\noindent
for $1 \leq p \leq r$; see [6, p.118]. We define a total order on the set of $2r$-tuples of integers by setting
\[
(a_1, b_1, \dotsc, a_r, b_r)  \trianglelefteq  (a_1', b_1', \dotsc, a_r', b_r')
\]

\noindent 
if $(b_r, a_r,  \dotsc, b_1, a_1)$ is less than or equal to $(b_r', a_r',  \dotsc, b_1', a_1')$ in the lexicographic order.

\begin{rmk}  %% 3.5
Note that the straightening relation in Theorem 3.3 (2) does not change the type of a given tableau. Hence for every even-tableau $T$, $[T]$ can be written as a linear sum $[T] = c_1 [T_1] + \cdots + c_k [T_k]$,  $c_i \in R$, where all $T_i$ are standard even-tableaux whose types are same as the type of $T$.
\end{rmk}

Since the coordinate ring $R[V]$ is a quotient ring of $R[Y_{ij}]_{i<j}$, we can naturally associate an even-tableau $T$ with an element of $R[V]$, which we also denote by $[T]$ by an abuse of notation. Note that $R[V]$ is a graded algebra over $R$. (The polynomials in $Y_{ij}$, $1 \leq i, j \leq n$, obtained from $(\ref{1.1})$ are all homogeneous.) We denote the degree $m$ homogeneous part of $R[V]$ by $R[V]_m$. Let $\mathcal{T}(m)$ be the set of all even-tableaux whose shapes are partitions of $2m$. Theorem 3.3 (1) implies the following:

\begin{cor}  %% 3.6
The set $\{[T] \in R[V] \; | \; T $ is a standard even-tableau in $\mathcal{T}(m) \}$ spans $R[V]_m$ over $R$.
\end{cor}

The proof of next propostion closely follows that of the Symplectic Carter-Lusztig lemma from [5, p.119].

\begin{prop}  %% 3.7
Let $F$ be a field with $\mbox{char } F = 0$ or $\mbox{char } F > r$. The set $\{[T] \; | \; T $ is a symplectic standard even-tableau$\}$ spans $F[V]$ over $F$.
\end{prop}

\begin{proof}
Let $m \geq 1$ be chosen aribitrarily. It suffices to prove that $F[V]_{m}$ is spanned by 
\[
\{[T] \; | \;  T \mbox{ is a symplectic standard even-tableau in } \mathcal{T}(m) \}.  \tag{3.3}  \label{3.3}
\] 
Let $S$ be the $F$-span of $(\ref{3.3})$, and suppose that $S$ is a proper subset of  $F[V]_{m}$ for a contradiction. Since the set of all $[T]$, $T \in \mathcal{T}(m)$, spans $F[V]_{m}$, the set $\{ T \in \mathcal{T}(m) \: | \: [T] \notin S \}$ is nonempty. We choose an even-tableau $T$ in this set such that the type of $T$ is as large as possible in the $\trianglelefteq$ ordering. By Remark 3.5, there must be a standard even-tableau $T'$ such that $[T'] \notin S$ and the type of $T'$ is same as the type of $T$. Replacing $T$ by $T'$, we may assume that $T$ is standard.

Since $[T] \notin S$, $T$ cannot be symplectic standard. Hence there is a position $(l, h)$ such that $T(l, h) = u$ or $\overline{u}$ with $u < h$. Assume that $h$ is minimal with this property. Say $T(l, h-1) = v$ or $\overline{v}$. Then 
$$h-1 \leq v \leq u < h.$$
Here $h-1 \leq v$ is by the minimality of $h$, and $v \leq u$ is by the standardness of $T$. Hence we are forced to have $h-1 = v = u$ and more precisely,
$$ T(l, h-1) = \overline{h-1} \quad \text{and} \quad T(l, h) = h-1. $$
Let $\bf{r}$ denote the set $\{ 1, 2, \dotsc, r \}$ and define
 \begin{flalign*}
&\Gamma = \{p \in \mathbf{r} \: | \text{  both } p \text{ and } \overline{p} \text{ occur among } T(l, 1), T(l, 2), \dotsc , T(l, h) \},\\
&A = \{ p \in \mathbf{r} \: | \text{ exactly one of } p \text{ or } \overline{p} \text{ occurs among } T(l, 1), T(l, 2), \dotsc , T(l, h) \},\\
&P' = \{ p \in \mathbf{r} \: | \: p  \text{ occurs in row } l \text{ of } T \} \smallsetminus \Gamma, \\
&Q' = \{ p \in \mathbf{r}  \: | \: \overline{p}  \text{ occurs in row } l \text{ of } T \} \smallsetminus \Gamma.
\end{flalign*}
Then row $l$ of $T$ corresponds to $ [P' \cup \Gamma, \overline{Q' \cup \Gamma}]$ up to sign, and $2 |\Gamma| + |A| = h$. From Proposition 2.8, we have
\[
[P' \cup \Gamma, \overline{Q' \cup \Gamma}] = (-1)^{ |\Gamma| } \sum_{\Gamma'} [P' \cup \Gamma', \overline{Q' \cup \Gamma'}] \tag{3.4} \label{3.4}
\]
\noindent where $\Gamma'$ runs over the subsets of $\mathbf{r} \smallsetminus (P' \cup Q' \cup \Gamma)$ such that $ |\Gamma'| = |\Gamma|$.
Any such $\Gamma'$ doesn't have intersection with $A \cup \Gamma$ (a disjoint union) and  $h = 2 |\Gamma| + |A| = |\Gamma'| + |\Gamma| + |A|$. Note that both $A$ and $\Gamma$ are subsets of $\{1, 2, \dotsc, h-1\}$. As a result, $\Gamma'$ must contain an element greater than $h-1$. Let $T_{\Gamma'}$ denote the tableau obtained from $T$ by replacing its row $l$ with the one row tableau corresponding to $[P' \cup \Gamma', \overline{Q' \cup \Gamma'}]$. Then $[T]$ is equal to 
\[
(-1)^{ |\Gamma| } \sum_{\substack {\Gamma' \subset r - P' \cup Q' \cup \Gamma \\ |\Gamma'| = |\Gamma|}} [T_{\Gamma'}]
\]
up to sign by $( \ref{3.4} )$. Since the type of each $T_{\Gamma'}$ is strictly bigger than the type of $T$ with respect to $\trianglelefteq$, all $[T_{\Gamma'}]$ are in $S$ by the maximality of the type of $T$. Now $[T]$ is also in $S$, a contradiction.
\end{proof}

\par
Now we prove the linear independency of the set $\{[T] : T \mbox{ is a symplectic standard even-tableau}\}$ over a field $F$ of arbitrary characteristic. First, let us specify some matrices in $Sp(n, F)$. Let $E_{i, j}$ denote the $n \times n$ matrix which has 1 at entry $(i, j)$ and 0 at all the other entries. It's easy to check that the matrices
\[
I_n -\mu \: E_{i, j}+ \mu \: E_{\overline{j}, \overline{i}}, \quad
I_n + \mu \: E_{\overline{i}, i}, \quad
I_n +\mu \: E_{\overline{j}, i}+ \mu \: E_{\overline{i}, j}
\]

\noindent
are in $Sp(n, F)$ if $\mu \in F$ and $i, j \in \{1, 2, \dotsc, r \}$ with $ i \neq j $.

\begin{defn}  %% 3.8
We say that a given tableau is \emph{canonical} if in each column $j$ of the tableau, only $\overline{j}$ appears. 
\end{defn}

\par
Here's an example of a canonical tableau:
\[
\begin{array}{cccc}
\overline{1} & \overline{2} &  \overline{3} & \overline{4}\\
\overline{1} & \overline{2} &   & \\
\overline{1} & \overline{2} &  & 
\end{array}
\]
Clearly every canonical tableau is symplectic standard.

\begin{prop}  %% 3.9
Let $F$ be a field of arbitrary characteristic.
\begin{enumerate}[(1)]
\item The map from the set of all canonical even-tableaux into $F[V]$, given by $T \mapsto [T]$, is injective and its image in $F[V]$ is linearly independent over $F$. 
\item The map from the set of all symplectic standard even-tableaux into $F[V]$, given by $T \mapsto [T]$, is injective and its image in $F[V]$ is linearly independent over $F$. 
\end{enumerate}
\end{prop}

\begin{proof}
Without loss of generality, we may assume that $F$ is infinite. First, we consider the case $r$ is even. For any $F$-algebra $C$, $V(C)$ contains all $n \times n$ matrices of the form
\[
\begin{tikzpicture}[
style1/.style={
  matrix of math nodes,
  every node/.append style={text width=#1,align=center,minimum height=5ex},
  nodes in empty cells,
  left delimiter=[,
  right delimiter=],
  }
]
\matrix[style1=0.5cm] (1mat)
{
  & & &\\
  & & &\\
  & & &\\
  & & &\\
};
\draw[dashed]
  (1mat-2-1.south west) -- (1mat-2-4.south east);
\draw[dashed]
  (1mat-1-2.north east) -- (1mat-4-2.south east);
\node[font=\large] 
  at (1mat-1-1.south east) {$0$};
\node[font=\large] 
  at (1mat-1-3.south east) {$0$};
\node[font=\large] 
  at (1mat-3-1.south east) {$0$};
\node[font=\large] 
  at (1mat-3-3.south east) {$A$};
\end{tikzpicture}
\]
\noindent
where $A$ is a $r \times r$ skew-symmetric matrix with entries in $C$. In other words, there is a closed embedding from the scheme of $r \times r$ skew-symmetric matrices into $V$. Let 
$$\psi: F[V] \rightarrow F[X_{ij}]_{i<j} := F[X_{i j}: 1 \leq i < j \leq r]$$
be the corresponding ring homomorphism. Note that every pfaffian in $F[V]$ of the form $[\bar{1}, \bar{2}, \dotsc, \bar{i}]$, $1 \leq i \leq r$, maps to the pfaffian $[1, 2, \dotsc, i]$ (evaluated at the generic $r \times r$ skew-symmetric matrix $X = (X_{ij})$) in $F[X_{ij}]_{i<j}$. Let $\mathbf{r} := \{ 1, 2, \dotsc, r \}$ and let $\eta: \mathbf{n} \rightarrow \mathbf{r}$ be a map sending $\overline{i}$ to $i$ for $1 \leq i \leq r$. For any canonical Tableau $T$ with entries in $\bf{n}$, $\eta \circ T$ is a tableau with entries in $\bf{r}$, and $[T] \in F[V]$ clearly maps to $[\eta \circ T] \in F[X_{ij}]_{i<j}$ under $\psi$. Moreover, if $T_1$ and $T_2$ are distinct canonical tableaux with entries in $\bf{n}$, then $\eta \circ T_1$ and $\eta \circ T_2$ are also distinct tableaux in the following set: 
\[
\left\lbrace T \; \middle| \;
\begin{tabular}{@{}l@{}}
$T$ is an even-tableau with entries in $\bf{r}$ such that the entries of $T$ strictly increase along each row\\
and weakly increase down each column with respect to the order $1 \leq 2 \leq \cdots \leq r$ 
\end{tabular}
\right\rbrace
\]
Then by Remark 3.4, the (composition) map from the set of all canonical even-tableaux into $F[X_{ij}]_{i<j}$, given by $T \mapsto \psi([T])$, is injective and its image in $F[X_{ij}]_{i<j}$ is linearly independent over $F$. This proves (1) when $r$ is even.

Now suppose $r$ is an odd integer. We consider the $n \times n$ matrices of the form
\[
\begin{tikzpicture}[
style1/.style={
  matrix of math nodes,
  every node/.append style={text width=#1,align=center,minimum height=3ex},
  nodes in empty cells,
  left delimiter=[,
  right delimiter=],
  }
]
\matrix[style1=0.3cm] (1mat)
{
  & & & & & & &\\
  & & & & & & &\\
  & & & & & & &\\
  & & & & & & &\\
  & & & & & & &\\
  & & & & & & &\\
  & & & & & & &\\
  & & & & & & &\\
};
\draw[dashed]
  (1mat-4-1.south west) -- (1mat-4-8.south east);
\draw[dashed]
  (1mat-1-4.north east) -- (1mat-8-4.south east);
\draw[dashed]
  (1mat-5-8.north west) -- (1mat-7-8.south west);
\draw[dashed]
  (1mat-7-4.south east) -- (1mat-7-8.south west);
\node[font=\large] 
  at (1mat-2-2.south east) {$0$};
\node[font=\large] 
  at (1mat-2-6.south east) {$0$};
\node[font=\large] 
  at (1mat-6-2.south east) {$0$};
\node[font=\large] 
  at (1mat-6-6.center) {$B$};
\node[font=\large] 
  at (1mat-8-5.east) {$0$};
\node[font=\large] 
  at (1mat-8-7.center) {$\cdots$};
\node[font=\large] 
  at (1mat-8-8.center) {$0$};
\node[font=\large] 
  at (1mat-6-8.south) {$\vdots$};
\node[font=\large] 
  at (1mat-5-8.south) {$0$};
\end{tikzpicture}
\]
denoted by $M(B)$, where $B$ is the submatrix consisting of the rows $\overline{1}, \overline{2}, \dotsc, \overline{r-1}$ and columns $\overline{1}, \overline{2}, \dotsc, \overline{r-1}$. When $C$ is an $F$-algebra, $V(C)$ contains all $M(B)$ such that $B$ is a $(r-1) \times (r-1)$ skew-symmetric matrix with entries in $C$. Thus, there is a closed embedding from the scheme of $(r-1) \times (r-1)$ skew-symmetric matrices into $V$. Since $r$ is odd, $\overline{r}$ never appears as an entry of $T$ if $T$ is a canonical even-tableau with entries in $\bf{n}$. The rest of the proof is analogous to the case $r$ is even.

The proof of (2) closely follows [9, pp.506-508]. 
Recall that $Sp(n, F)$ acts on $V$: if $g \in Sp(n, F)$ and $Y \in V(C)$, $C$ is any $F$-algebra, (right) action is given by $Y \cdot g = g^{T}Yg$. There is an induced action of $Sp(n, F)$ on $F[V]$, which can be described as
$$(g \cdot f) (Y) = f (Y \cdot g) = f (g^{T}Yg)$$
for $g \in Sp(n, F)$ and $f \in F[V]$. For example, the action of $I_n + \mu \: E_{\overline{i}, i}$ ($\mu \in F$) on the matrix $Y$ is the transformation $i \rightarrow i + \mu \: \overline{i}$ applied to $Y$ on both rows and columns. Then the result of the induced action of  $I_n + \mu \: E_{\overline{i}, i}$ on $[i_1, i_2, \dotsc, i_l] \in F[V]$ depends on whether $i$ or $\overline{i}$ appear among the indices $i_1, i_2, \dotsc, i_l$. More precisely,
\begin{flalign*}
&(I_n + \mu \: E_{\overline{i}, i}) \cdot [i_1, i_2, \dotsc, i_l]\\ 
& \hspace{20mm} = 
\begin{cases}
[i_1, i_2, \dotsc, i_l] + \mu \: [i_1, i_2, \dotsc, i_{k-1}, \overline{i}, i_{k+1}, \dotsc, i_l] 
& \mbox{if } i = i_k  \mbox{ and } \overline{i} \mbox{ does not appear,} \\
[i_1, i_2, \dotsc, i_l] & \mbox{if } i \mbox{ does not appear,}\\
[i_1, i_2, \dotsc, i_l] & \mbox{if both } i \mbox{ and } \overline{i} \mbox{ appear.}
\end{cases}
\end{flalign*}

\noindent
Finally, if $I_n + \mu \: E_{\overline{i}, i}$ acts on $[T] \in F[V]$ for some even-tableau $T$, then we get a polynomial in $\mu$ (with coefficients in $F[V]$) of degree equal to the number of the rows of $T$ in which $i$ appears and $\overline{i}$ does not appear. The leading coefficient of this polynomial is $[T']$, where $T'$ is a tableau obtained from $T$ by replacing $i$ with $ \overline{i}$ in every row of $T$ containing $i$ but not $\overline{i}$.

Now assume that there is a dependence relation over $F$,
\[
\sum_{i=1}^{s} c_i [T_i] = 0,  \tag{3.5}  \label{3.5} 
\]
where the $T_i$ are distinct symplectic standard even-tableaux. Some $T_i$ is not canonical by (1), and therefore $T_i(u, v) \neq \overline{v}$ must hold for some $i$ and some $u, v$. Let $p$ be the entry $T_i(u, v)$ as small as possible with this property (with respect to the order $( \ref{3.1} )$), i.e.
$$p := \text{min} \{ T_i (u, v) \: | \: \exists \text{ a tableau } T_i \text{ in } (\ref{3.5})  \text{ and a position } (u, v) \text{ such that } T_i (u, v) \neq \overline{v} \}.$$
Then we define $j$ and $h$ in the following way:
 \begin{flalign*}
& j := \text{the minimum of the column indices where } p \text{ appear.}\\
&\hspace{2.5mm} =  \text{min} \{v \: | \: \exists \text{ a tableau } T_i \text{ and a position } (u, v) \text{ such that } T_i (u, v) = p \}\\
& h := \text{the maximum number of occurrences of } p \text{ in the column } j \text{ of some } T_i .
\end{flalign*}

\noindent
By rearranging the $T_i$, we may assume that $T_1, T_2, \dotsc, T_k$ $(k \leq s)$ are the tableaux which have the entry $p$ exactly $h$ times in their column $j$ (necessarily in consecutive rows). We divide into three cases.\\

\noindent \textit{Case}
$p \in \{ \overline{1}, \overline{2}, \dotsc, \overline{r} \}$: Say $p = \overline{q}$ and let 
$ I_n -\mu \: E_{q, j}+ \mu \: E_{\overline{j}, \overline{q}}$
 act on $\sum_{i=1}^{s} c_i [T_i]$. The action of
$\: I_n -\mu \: E_{q, j}+ \mu \: E_{\overline{j}, \overline{q}}$
on the generic matrix $Y$ is the transformation $j \rightarrow j - \mu \: q$ and $p \rightarrow p + \mu \: \overline{j}$ applied to $Y$ on both rows and columns. Since $j$ never appears as an entry in all tableaux $T_i$, $1 \leq i \leq s$, by the symplectic standardness of each $T_i$ and the minimality of $p$, only the transformation $p \rightarrow p + \mu \: \overline{j}$ matters.\\

\noindent \textit{Case}
$p = j$: Let $I_n + \mu \: E_{\overline{j}, j}$ act on $\sum_{i=1}^{s} c_i [T_i]$. Note that the action of
$I_n + \mu \: E_{\overline{j}, j}$
on the generic matrix $Y$ is the transformation $j \rightarrow j + \mu \: \overline{j}$ applied to $Y$ on both rows and columns. \\

\noindent \textit{Case}
$p \in \{ 1, 2, \dotsc, r \}$ with $p > j$: Let $I_n +\mu \: E_{\overline{j}, p}+ \mu \: E_{\overline{p}, j}$ act 
on $\sum_{i=1}^{s} c_i [T_i]$. The action of
$I_n +\mu \: E_{\overline{j}, p}+ \mu \: E_{\overline{p}, j}$
on the generic matrix $Y$ is the transformation $p \rightarrow p + \mu \: \overline{j}$ and $j \rightarrow j + \mu \: \overline{p}$ applied to $Y$ on both rows and columns. Again $j$ never appears as an entry in all tableaux $T_i$, so only the transformation $p \rightarrow p + \mu \: \overline{j}$ matters.\\

\noindent
In each case, the action on $\sum_{i=1}^{s} c_i [T_i]$ results in a polynomial in $\mu$ of degree $h$, whose leading coefficient is 
$\sum_{i=1}^{k} c_i [T_i ']$ 
where $T_i '$ is obtained from $T_i,\: 1 \leq i \leq k,$ by substituting $ \overline{j}$ for the entries $p$ in the column $j$. This polynomial has value 0 in $F[V]$ for all $\mu \in F$, so the leading coefficient must vanish by the following Lemma 3.10. It's not difficult to check that the tableaux $T_1', \dotsc, T_k'$ are symplectic standard and distinct from each other. This new relation $\sum_{i=1}^{k} c_i [T_i '] = 0$ has either bigger $p$ or same $p$ with bigger $j$ compared to the original relation $\sum_{i=1}^{s} c_i [T_i] = 0$. Hence, we are in an inductive procedure on $(p, j)$ which ends with a relation in which only canonical tableaux are involved. It contradicts (1) so we get the desired conclusion.
\end{proof}

\begin{lem}  %% 3.10
Assume $F$ is an infinite field and let $A$ be any $F$-algebra. If a polynomial $a_n x^n + \cdots + a_1 x + a_0 \in A[x]$ has more than $n$ distinct roots in $F$, then $a_n$ must be zero. (In fact, all $a_i$ are then zero by iterated use of this lemma.)
\end{lem}

\begin{proof}
We induct on the degree of the polynomial. Suppose $a_1 x + a_0 \in A[x]$ has two distinct roots $x_1, x_2 \in F$. It follows that
\[
0 = (a_1 x_1 + a_0) - (a_1 x_2 + a_0) = a_1 (x_1 - x_2).
\]
Since $x_1 - x_2 \in F^{\times}$, clearly $a_1$ is equal to zero. \par
Now assume that the lemma holds true for all degree $k$ polynomials in $A[x]$. Suppose that a polynomial $a_{k+1} x^{k+1} + \cdots + a_1 x + a_0 \in A[x]$ has more than $k+1$ distinct roots in $F$, and pick a root $\alpha \in F$. Then
\begin{flalign*}
a_{k+1} x^{k+1} + \cdots + a_2 x^2 + a_1 x + a_0 
&= a_{k+1} x^{k+1} + \cdots +  a_2 x^2 + a_1 x + (-a_{k+1} {\alpha}^{k+1} - \cdots -  a_2 \alpha^2 - a_1 \alpha)\\
&= a_{k+1} (x^{k+1}-\alpha^{k+1}) + \cdots + a_2 (x^2 - \alpha^2) + a_1 (x - \alpha)\\
&= (x-\alpha) (a_{k+1} \frac{x^{k+1}-\alpha^{k+1}}{x-\alpha} + \cdots + a_2 \frac{x^2 - \alpha^2}{x - \alpha} + a_1).
\end{flalign*}
Note that the second factor is a polynomial of degree $k$ with leading coefficient $a_{k+1}$. For any root $\beta$ of $a_{k+1} x^{k+1} + \cdots + a_1 x + a_0$ with $\beta \in F, \beta \neq \alpha$, evidently $\beta$ must also be a root of the second factor. Hence this degree $k$ polynomial has more than $k$ distinct roots in $F$, and therefore $a_{k+1} = 0$ by the induction hypothesis.
\end{proof}

Combining Proposition 3.7 and Proposition 3.9\:(2), we come to the following conclusion:

\begin{thm}  %% 3.11
When $F$ is a field with $\mbox{char } F = 0$ or $\mbox{char } F > r$, there is an $F$-basis for $F[V]$ indexed by the symplectic standard even-tableaux.
\end{thm}

\begin{cor}  %% 3.12
Over any $F$-algebra $R$, there is an $R$-basis for $R[V]$ indexed by the symplectic standard even-tableaux.
\end{cor}

\begin{proof}
It is an easy consequence of Theorem 3.11 and the base change from $F$ to $R$.
\end{proof}

\begin{rmk}  %% 3.13
When $V$ is defined over a field $F$ with $\mbox{char } F = 0$ or $\mbox{char } F > r$, the condition $\mbox{char}_{(-JY)}(T) = T^n$ imposed on V can be replaced by the weaker condition 
\[ 
\mbox{tr}(-JY) = 2 \sum_{i=1}^r Y_{i \overline{i}} = 0. \tag{3.6} \label{3.6}
\]
Let $V'$ be the scheme of $n \times n$ matrices $Y$ defined by the conditions $Y = -Y^T, \: Y^T JY = 0,$ and $( \ref{3.6})$. In Section 2 and Section 3, we never used the condition $\mbox{char}_{(-JY)}(T) = T^n$ itself except for the weaker one $( \ref{3.6})$. Hence if we start with $V'$ instead of $V$, we can show that there is an $F$-basis for $F[V']$ indexed by the symplectic standard even-tableaux by the same argument. Then it's clear that $F[V'] = F[V]$.
\end{rmk}

When $R$ is $\mathbb{Z}$ or $\mathbb{Q}$, let $I_R$ denote the ideal of the polynomial ring $R [Y_{ij} : 1 \leq i, j \leq n]$ generated by the conditions $(\ref{1.1})$, respectively. By definition, $R[V] = R[Y_{ij}] \big/ I_R$. It's not difficult to see that a polynomial $f \in \mathbb{Z}[Y_{ij}]$ is in $I_\mathbb{Q}$ if and only if for some nonzero integer $m$, $mf$ is in $I_\mathbb{Z}$. In other words, if we consider $\mathbb{Z}[V]$ as a $\mathbb{Z}$-module, $I_\mathbb{Q} \cap \mathbb{Z}[Y_{ij}] \big/ I_\mathbb{Z}$ is the torsion submodule of $\mathbb{Z}[V]$, and the quotient module 
\[
S :=  \mathbb{Z}[V] \big/ (I_\mathbb{Q} \cap \mathbb{Z}[Y_{ij}] \big/ I_\mathbb{Z})  \cong \mathbb{Z}[Y_{ij}] \big/ I_\mathbb{Q} \cap \mathbb{Z}[Y_{ij}]
\]
is torsion-free. Note that the morphism $S \rightarrow \mathbb{Q}[V]$ (induced from $\mathbb{Z}[Y_{ij}] \rightarrow \mathbb{Q}[Y_{ij}]$) is injective, and hence we can regard  $S$ as a subring of $\mathbb{Q}[V]$.

\begin{prop}  %% 3.14
There is a $\mathbb{Z}$-basis for $S$ indexed by the symplectic standard even-tableaux.
\end{prop}

\begin{proof}
It suffices to show that $S$ is the $\mathbb{Z}$-span of the set $\{[T] \in \mathbb{Q}[V] : T \text{ is a symplectic standard} \text{ even-tableau}\}$ in $\mathbb{Q}[V]$. Pick any nonzero $f \in S$. By Theorem 3.11, we can write $f$ as
\[
 f  = \sum_{i = 1}^{k} c_i [T_i],
\]
\noindent
where the $T_i$ are symplectic standard even-tableaux and coefficients $c_i$ are taken in $\mathbb{Q}$. We show that $c_i$ are actually in $\mathbb{Z}$. Express each $c_i$ in a reduced form $c_i = a_i / b_i$ with $a_i, b_i \in \mathbb{Z}, b_i \neq 0$, and let $l$ be the least common multiple of $b_1, b_2, \cdots, b_k$. Multiplying $l$ to both sides of the equation, we get an expression with coefficients in $\mathbb{Z}:$
\[
lf  = \sum_{i = 1}^{k} (l c_i) [T_i].
\]

\noindent
 Suppose $l>1$ and pick a prime factor $p$ of $ l$. Modulo $p$ we get a nontrivial relation, which contradicts Proposition 3.9\:(2). Hence $l$ must be equal to $1$, and this proves the proposition.
\end{proof}

When $F = \mathbb{C}$, there is an alternative way to prove Proposition 3.9 (2). For any even-tableau $T$, the element $[T]$ is  homogeneous in the graded ring $\mathbb{C}[V]$. Thus, it suffices to prove that the map from the set 
\[
\{T \; | \; T \text{ is a symplectic standard even-tableau whose shape is a partition of } 2m \}.  \tag{3.7}  \label{3.7}
\] 
 into $F[V]_m$, given by $T \mapsto [T]$, is injective and its image in $F[V]_m$ is linearly independent over $F$. 
Given a partition $\lambda$ (shape of a tableau), we use the following notations:
\begin{flalign*}
&\mbox{row}(i, \lambda) := \mbox{ the size of row } i \mbox{ of } \lambda \mbox{ and possibly equal to zero if }  \lambda  \mbox{ does not have row } i,\\
&\mbox{col}(j, \lambda) := \mbox{ the size of column } j \mbox{ of } \lambda \mbox{ and possibly equal to zero if } \lambda  \mbox{ does not have column } j.
\end{flalign*}

\begin{prop}  %% 3.15
Let $\mathcal{P}_m$ be the set of all partitions $\lambda$ of 2m whose row$(i, \lambda)$ is even and $\leq r$ for all $i$. The degree $m$ homogeneous part of $\mathbb{C} [V]$ can be decomposed as
\[
\mathbb{C} [V]_m = \bigoplus_{\lambda \in \mathcal{P}_m} L_{\lambda} E,
\]
\noindent
where $L_{\lambda} E$ denotes the irreducible representation of $Sp(n, \mathbb{C})$ whose highest weight is 
$$ col(1, \lambda), col(2, \lambda), \dotsc, col(n, \lambda). $$
\end{prop}

\begin{proof}
We first show the inclusion
\[
\mathbb{C} [V]_m \; \supseteq \bigoplus_{\lambda  \in \mathcal{P}_m} L_{\lambda} E,
\]
and the proof is similar to that of [10, Prop.2.3.8\:(b)]. For each even number $s,\: 0 \leq s \leq r$, we define $g_{s} := [1, 2, \dotsc, s] \in \mathbb{C} [V]$. Let U be the subgroup of $Sp(n, \mathbb{C})$ consisting of all upper triangular matrices with 1's on the diagonal. It's immediate that for each $s$, $g_s$ is a $U$-invariant of the weight $(1^s, 0^{n-s})$. Now for any partition $\lambda \in \mathcal{P}_m$, we see that $g_{\lambda} := \prod_i g_{row(i, \lambda)}$ is a nonzero $U$-invariant of the weight
$$ col(1, \lambda), col(2, \lambda), \dotsc, col(n, \lambda), $$
and therefore $\mathbb{C} [V]_m \; \supseteq  L_{\lambda} E$. 
(By a slight modification of the proof of Proposition 3.9(1), we can show that $g_\lambda$ is nonzero in $\mathbb{C} [V]$.) \par

From [5], we know that
\begin{center}
$dim_{\mathbb{C}}(L_{\lambda} E)$ = the number of symplectic standard even-tableaux of shape $\lambda$
\end{center}
\noindent
 for every $\lambda \in \mathcal{P}_m$. Note that if $T$ is in $(\ref{3.7})$, then the shape of $T$ is in $ \mathcal{P}_m$. By the proof of Proposition 3.7, the set $(\ref{3.3})$ spans $\mathbb{C} [V]_m$. It follows that
\[
dim_{\mathbb{C}}(\mathbb{C} [V]_m) \leq \sum_{\lambda \in \mathcal{P}_m} dim_{\mathbb{C}}(L_{\lambda} E),
\]
and therefore the equality
\[
\mathbb{C} [V]_m = \bigoplus_{\lambda \in \mathcal{P}_m} L_{\lambda} E. 
\]
This also proves that the set $(\ref{3.3})$ is a $\mathbb{C}$-basis of $\mathbb{C} [V]_m$.
\end{proof}

%%%%%%%%%%%%%%%%%%%%%%%%%%%%%%%%%%%%%%%%%%%%%%%%%%%%%%%%%%%%%%%%%%%%%%%%%%%%%%%%%%%%%%%%%%%%%%%%%%%%%%%%%%%%%%%%

\section{Reducedness of the coordinate ring}

We assume that $n=2r$ is a multiple of $4$ and that $F$ is a field with $\mbox{char } F = 0$ or $\mbox{char } F > r$. Then Remark 3.11 implies 
\[
F[V] = F[Y_{i j} : 1 \leq i, j \leq n] \big/ ( Y^{T} JY, \: Y+Y^T, \: \sum_{i=1}^{r} Y_{i \overline{i}} ),
\]
where $Y = (Y_{ij})$ denotes $n \times n$ generic matrix. We define $f$ as the minor of $Y$ consisting of the rows  $\overline{1}, \overline{2}, \dotsc, \overline{r}$ and columns $\overline{1}, \overline{2}, \dotsc, \overline{r}$.

\begin{lem}  %% 4.1
The element $f$ is not a zero divisor in $F[V]$.
\end{lem}

\begin{proof}
This follows the proof of [6, Cor.4.4]. It suffices to show that the pfaffian $[\overline{1}, \overline{2}, \dotsc, \overline{r}]$ is not a zero divisor in $F[V]$ since $f$ is the square of $[\overline{1}, \overline{2}, \dotsc, \overline{r}]$. ($[\overline{1}, \overline{2}, \dotsc, \overline{r}] = 0$ if $r$ is odd, so we need the assumption that $r$ is even.) The one row tableau associated to $[\overline{1}, \overline{2}, \dotsc, \overline{r}]$ is obviously symplectic standard. Furthermore, for any symplectic standard even-tableau $T$, the product
$$ [\overline{1}, \overline{2}, \dotsc, \overline{r}] \cdot [T] $$
\noindent
is again associated to a symplectic standard even-tableau. Since the elements of $F[V]$ indexed by the symplectic standard even-tableaux form a basis of $F[V]$, this clearly implies that $[\overline{1}, \overline{2}, \dotsc, \overline{r}]$ is not a zero divisor.
\end{proof}

\begin{lem}  %% 4.2
The scheme $\text{Spec} \: F[V]_f$ is isomorphic to an open subscheme of $\mathbb{A}^{r^2}$.
\end{lem}

\begin{proof}
Consider two affine open subschemes of the Grassmannian scheme $Gr(n, 2n)$ where we can represent the $n$-dimensional subspaces as 
\[
\renewcommand{\arraystretch}{1.5}
M =
\left[
\begin{array}{ccc}
Y_{1, 1} &  \cdots & Y_{1, n} \\
\vdots &  & \vdots \\
Y_{n, 1} & \cdots & Y_{n, n} \\
\hline
1 & & \\
& \ddots & \\
& & 1 \\
\end{array}
\right], \\
\quad \quad
N = 
\left[
\begin{array}{ccc}
Z_{1, 1} &  \cdots & Z_{1, n} \\
\vdots &  & \vdots \\
Z_{r, 1} & \cdots & Z_{r, n} \\
\hline
1 & & \\
& \ddots & \\
& & 1 \\
\hline
Z_{r+1, 1} &  \cdots & Z_{r+1, n} \\
\vdots &  & \vdots \\
Z_{n, 1} & \cdots & Z_{n, n} \\
\end{array}
\right].
\]

\noindent
Let $H$ be the $n \times n$ submatrix of  $M$ given by

\[
\renewcommand{\arraystretch}{1.5}
H =
\left[
\begin{array}{cccccc}
Y_{r+1, 1} &  \cdots & Y_{r+1, r} & Y_{r+1, r+1} &  \cdots & Y_{r+1, n}  \\
\vdots &  & \vdots & \vdots &  & \vdots \\
Y_{n, 1} &  \cdots & Y_{n, r} & Y_{n, r+1} &  \cdots & Y_{n, n}  \\
\hline
1 & & & & & \\
& \ddots & & & &\\
& & 1 & & &\\
\end{array}
\right],
\]
\noindent
and $G$ be the $n \times n$ submatrix of $N$ defined as
\[
\renewcommand{\arraystretch}{1.5}
G =
\left[
\begin{array}{cccccc}
& & & 1& & \\
& & & & \ddots &\\
& & & & & 1 \\
\hline
Z_{r+1, 1} &  \cdots & Z_{r+1, r} & Z_{r+1, r+1} &  \cdots & Z_{r+1, n}  \\
\vdots &  & \vdots & \vdots &  & \vdots \\
Z_{n, 1} &  \cdots & Z_{n, r} & Z_{n, r+1} &  \cdots & Z_{n, n}  \\
\end{array}
\right].
\]
\noindent
Let $h$ and $g$ denote the determinants of $H$ and $G$ respectively. We have a ring isomorphism between localizations $F [Y_{i, j}]_{h}$ and  $F [Z_{i, j}]_{g}$ matching $Y_{i, j}$ to the entry $(i, j)$ of the following product of matrices:

\[
\renewcommand{\arraystretch}{1.5}
\left[
\begin{array}{cccccc}
Z_{1, 1} &  \cdots & Z_{1, r} & Z_{1, r+1} &  \cdots & Z_{1, n}  \\
\vdots &  & \vdots & \vdots &  & \vdots \\
Z_{r, 1} &  \cdots & Z_{r, r} & Z_{r, r+1} &  \cdots & Z_{r, n}  \\
\hline
1 & & & & & \\
& \ddots & & & &\\
& & 1 & & &\\
\end{array}
\right] \\
\left[
\begin{array}{cccccc}
& & & 1& & \\
& & & & \ddots &\\
& & & & & 1 \\
\hline
Z_{r+1, 1} &  \cdots & Z_{r+1, r} & Z_{r+1, r+1} &  \cdots & Z_{r+1, n}  \\
\vdots &  & \vdots & \vdots &  & \vdots \\
Z_{n, 1} &  \cdots & Z_{n, r} & Z_{n, r+1} &  \cdots & Z_{n, n}  \\
\end{array}
\right]^{-1}.
\]
\noindent
(The morphism $F [Y_{i, j}]_{h} \xrightarrow{\sim} F [Z_{i, j}]_{g}$ is a glueing datum of the affine open subschemes of $Gr (n, 2n)$.)

Remember that $I_n$ denotes $n \times n$ identity matrix and $O_n$ denotes $n \times n$ zero matrix. Consider the following conditions given to a $2n \times n$ matrix $U$:
\begin{enumerate}[(1)]
\item $ \big( \left[ \begin{array}{c|c} I_{n} & O_{n} \end{array} \right] U \big)^{T} 
J \:
\big( \left[ \begin{array}{c|c} I_{n} & O_{n} \end{array} \right] U \big) = O_n$,

\item $U^T
\left[
\begin{array}{c|c}
O_n & \hspace{2mm} I_{n} \\
\hline
  I_{n} & O_n
\end{array}
\right]
U = O_n$.
\end{enumerate}
\ \\

%% Whenever $U$ satisfies  (1) and (2), so does $U \cdot A$ for any $n \times n$ matrix $A$. In fact, these conditions define a closed subscheme of the Grassmannian scheme $Gr(n, 2n)$.

\noindent \textit{Case} 
$U = M$: the conditions are equivalent to
\begin{enumerate}[(1)]
\item $Y^T J \: Y =  O_n,$
\item $Y + Y^T = O_n,$
\end{enumerate}

\noindent
where $Y = (Y_{i, j})$ denotes the $n \times n$ submatrix of $M$. Let $\alpha$ be the ideal of $F [Y_{i, j}]$ generated by these conditions on $Y$.\\

\noindent \textit{Case} 
$U = N$: for convenience, we divide the $n \times n$ matrix $( Z_{i, j} )$ into four $r \times r$ blocks 
\[
\left[
\begin{array}{c|c}
A & B \\
\hline
C & D
\end{array}
\right].
\]
\noindent
Then equivalent conditions are
\begin{enumerate}[(1)]
\item $A^T = A, \quad B = O_r$,
\item $C^T = -C, \quad B^T = -B, \quad D = -A^T$,
\end{enumerate}

\noindent
so (1) and (2) together:
\[
A^T = A, \quad B = O_r, \quad C^T = -C, \quad D = -A^T.  \tag{4.1}  \label{4.1}
\]

\noindent
When $\beta$ is the ideal of $F [Z_{i, j}]$ generated by $( \ref{4.1} )$, it's easy to see that ${F [Z_{i, j}]} \big/ {\beta}$ is isomorphic to the polynomial ring in indeterminates $\{A_{i, j}\}_{i \leq j}$ and $\{C_{i, j}\}_{i < j}$. That is, $\mbox{Spec} \: \big( {F [Z_{i, j}]} \big/ {\beta} \big) \cong \mathbb{A}^{r^2}$.\\

We claim that the morphism $F [Y_{i, j}]_{h} \xrightarrow{\sim} F [Z_{i, j}]_{g}$ induces an isomorphism between quotients 
\[
 (F [Y_{i, j}] \big/ \alpha)_{h} \quad \cong \quad F [Y_{i, j}]_{h} \big/ \alpha_{h} \quad \rightarrow \quad F [Z_{i, j}]_{g} \big/ \beta_{g} \quad \cong \quad (F [Z_{i, j}] \big/ \beta)_{g}.  \tag{4.2} \label{4.2}
\]

\noindent
Generators of $\alpha$ from the condition $Y^T J Y = O_n$ are entries of the product of matrices
\[
\big( \left[ \begin{array}{c|c} I_{n} & O_{n} \end{array} \right] M \big)^{T} 
J \:
\big( \left[ \begin{array}{c|c} I_{n} & O_{n} \end{array} \right] M \big),
\]

\noindent
and under the map $F [Y_{i, j}]_{h} \rightarrow F [Z_{i, j}]_{g}$  each entry maps to the entry at the exact same location of the product of matrices
\[
\quad \big( \left[ \begin{array}{c|c} I_{n} & O_{n} \end{array} \right] N G^{-1} \big)^{T} 
J \:
\big( \left[ \begin{array}{c|c} I_{n} & O_{n} \end{array} \right] N G^{-1} \big),
\]

\noindent
which is equal to
\[
(G^{-1})^T  \big( \left[ \begin{array}{c|c} I_{n} & O_{n} \end{array} \right] N \big)^{T} 
J \:
\big( \left[ \begin{array}{c|c} I_{n} & O_{n} \end{array} \right] N \big) G^{-1}.
\]

\noindent
Note that entries of the product
\[
\big( \left[ \begin{array}{c|c} I_{n} & O_{n} \end{array} \right] N \big)^{T} 
J \:
\big( \left[ \begin{array}{c|c} I_{n} & O_{n} \end{array} \right] N \big)
\]

\noindent
are in $\beta$. It follows that generators of $\alpha$ from the first condition map to elements of $\beta_g$. Likewise, the generators of $\alpha$ obtained from $Y + Y^T = O_n$ are entries of the product

$$M^T 
\left[
\begin{array}{c|c}
O_n & \hspace{2mm} I_{n} \\
\hline
  I_{n} & O_n
\end{array}
\right]
M,$$

\noindent
which map to the entries at the same location of
\[
(N G^{-1})^T
\left[
\begin{array}{c|c}
O_n & \hspace{2mm} I_{n} \\
\hline
  I_{n} & O_n
\end{array}
\right]
N G^{-1}
\quad
=
\quad
(G^{-1})^T \{ N^T 
\left[
\begin{array}{c|c}
O_n & \hspace{2mm} I_{n} \\
\hline
  I_{n} & O_n
\end{array}
\right]
N \} G^{-1}.
\]

\noindent
These entries are in $\beta_g$, so we know that the morphism $( \ref{4.2} )$ is well-defined. In the same way, we can show that generators of $\beta$ map to elements of $\alpha_h$ under the inverse morphism 
$ F [Z_{i, j}]_{g} \rightarrow  F [Y_{i, j}]_{h}$. This proves that the morphism $( \ref{4.2} )$ is isomorphic.

Finally, we consider the last condition $\sum_{i=1}^r Y_{i \overline{i}} = 0$ imposed on V. We're interested in the image of  $\sum_{i=1}^r Y_{i \overline{i}}$ under the morphism $( \ref{4.2} )$. It's not difficult to see that
\[
 \sum_{i=1}^r Y_{i \overline{i}} = \mbox{tr} 
\renewcommand{\arraystretch}{1.5}
\left( \left[
\begin{array}{ccc}
Y_{1, \overline{1}} &  \cdots & Y_{1, \overline{r}} \\
\vdots &  & \vdots \\
Y_{r, \overline{1}} & \cdots & Y_{r, \overline{r}} \\
\end{array}
\right] \right)
\]

\noindent
maps to $\mbox{tr} (AC^{-1})$ under the morphism $F [Y_{i, j}]_{h} \xrightarrow{\sim} F [Z_{i, j}]_{g}$. Let $C^{i,j}$ denote the cofactor of the $(i, j)$ entry of the matrix $C$. Then
\begin{flalign*}
\mbox{tr} (AC^{-1}) &= \frac{1}{det(C)} \sum_{i=1}^{r} \sum_{j=1}^{r} A_{i, j} C^{i, j} \\
&= \frac{1}{det(C)} \{  \sum_{i=1}^{r} A_{i, i} C^{i, i} + \sum_{i < j} (A_{i, j} C^{i, j} + A_{j, i} C^{j, i}) \}.
\end{flalign*}

\noindent
Recall that $r$ is an even number. Since $C$ is a skew-symmetric $r \times r$ matrix, $C^{i, j} = -C^{j, i}$ if $i < j$ and $C^{i, i} = 0$ for all $i$. Furthermore the matrix $A$ is symmetric, so we have $\mbox{tr} (AC^{-1}) = 0$ in $F [Z_{i, j}]_{g} \big/ \beta_{g}$. In conclusion,
\[
F[V]_{f} \cong \big( F [Y_{i, j}] \big/ \alpha + (\sum_{i=1}^r Y_{i \overline{i}}) \big)_h \cong (F [Z_{i, j}] \big/ \beta)_{g}.  \tag{4.3} \label{4.3} \]

\noindent
Since $\mbox{Spec} \: (F [Z_{i, j}] \big/ \beta) \cong \mathbb{A}^{r^2}$, it proves the lemma.
\end{proof}

\begin{thm}  %% 4.3
The coordinate ring $F[V]$ is an integral domain.
\end{thm}

\begin{proof}
From the isomorphism $( \ref{4.3} )$, we know that $F[V]_f$ is an integral domain. Then by Lemma 4.1, there is an embedding $F[V] \rightarrow F[V]_f$, and hence $F[V]$ must also be an integral domain.
\end{proof}

%%%%%%%%%%%%%%%%%%%%%%%%%%%%%%%%%%%%%%%%%%%%%%%%%%%%%%%%%%%%%%%%%%%%%%%%%%%%%%%%%%%%%%%%%%%%%%%%%%%%%%%%%%%%%%%%

%%%%%%%%%%%%%%%%%%%%%%%%%%%%%%%%%%%%%%%%%%%%%%%%%%%%%%%%%%%%%%%%%%%%%%%%%%%%%%%%%%%%%%%%%%%%%%%%%%%%%%%%%%%%%%%

\end{document}